\documentclass[12pt]{amsart}
\usepackage{amssymb}

\usepackage{amsfonts}
\usepackage{geometry}
\usepackage[doublespacing]{setspace}
\usepackage[normalem]{ulem}

\newtheorem{theorem}{Theorem}
\theoremstyle{plain}

\newtheorem{definition}{Definition}

\newtheorem{lemma}{Lemma}

\newtheorem{proposition}{Proposition}

\numberwithin{equation}{section}
\input{tcilatex}

\begin{document}
\title{On the Applications of Cyclotomic Fields in Introductory Number Theory%
}
\author{Kabalan Gaspard}
\date{June 22, 2011, re-edited February 11, 2012}
\maketitle

\begin{abstract}
In this essay, we see how prime cyclotomic fields (cyclotomic fields
obtained by adjoining a primitive $p$-th root of unity to $%
\mathbb{Q}
$, where $p$ is an odd prime) can lead to elegant proofs of number
theoretical concepts. We namely develop the notion of primary units in a
cyclotomic field, demonstrate their equivalence to real units in this case,
and show how this leads to a proof of a special case of Fermat's Last
Theorem. We finally modernize Dirichlet's solution to Pell's Equation.
\end{abstract}

\bigskip Throughout this paper, unless specified otherwise, $\zeta \equiv
\zeta _{p}\equiv e^{\frac{2\pi \sqrt{-1}}{p}}$ where $p$ is an odd prime. $%
K\equiv 
\mathbb{Q}
(\zeta )$ and $\mathcal{O}_{K}$ is the ring of integers of $K$. We assume
knowledge of the basic properties of prime cyclotomic fields that can be
found in any introductory algebraic number theory textbook, namely that:

\begin{itemize}
\item $Gal(K:%
\mathbb{Q}
)\simeq U(%
\mathbb{Z}
/p%
\mathbb{Z}
)$ (the group of units of $%
\mathbb{Z}
/p%
\mathbb{Z}
$), which is cyclic and of order $p-1$.

\item $\mathcal{O}_{K}=%
\mathbb{Z}
\lbrack \zeta _{p}]=\left\langle 1,\zeta _{p},...,\zeta
_{p}^{p-2}\right\rangle _{%
\mathbb{Z}
}$, where $\left\{ 1,\zeta _{p},...,\zeta _{p}^{p-2}\right\} $ is a $%
\mathbb{Z}
$-basis for $\mathcal{O}_{K}$.

\item The only roots of unity in $\mathcal{O}_{K}$ (i.e. solutions in $%
\mathbb{C}
$ to $x^{n}=1$ for some $n\in 
\mathbb{N}
$) are of the form $\pm \zeta _{p}^{i}$, $i\in 
\mathbb{Z}
$.
\end{itemize}

We also assume elementary knowledge of quadratic characters, quadratic
reciprocity, and the Legendre symbol $\left( \dfrac{k}{p}\right) $.

\bigskip

\section{Primary elements in $\mathcal{O}_{K}$}

\begin{definition}
\label{PrimaryDef}Let $\alpha \in \mathcal{O}_{K}$ with $\alpha $ prime to $p
$. Then $\alpha $ is \emph{primary} iff $\alpha $ is congruent to a rational
integer modulo $(1-\zeta _{p})^{2}$.
\end{definition}

The definition of primary elements has historically been ambiguous in Number
Theory. In \cite{Dalawat}, Dalawat shows that definitions of primary
elements in $\mathcal{O}_{K}$ even differ by country ("$p$-primary", "%
\textit{primaire}" and "\textit{prim\"{a}r}") and, even though these
definitions do form a chain of implications, they are not equivalent.

We also note that it is not true that if $p$ an arbitrary odd prime and $\mu 
$ prime in $\mathcal{O}_{K}$, only one associate of $\mu $ is primary (for
example, according to the above definition, both $\pm (4+3\omega )$ are
primary in the ring of integers of $%
\mathbb{Q}
(\omega )$ where $\omega =e^{\frac{2\pi \sqrt{-1}}{3}}$).

\begin{proposition}
\label{Only1PrimaryGeneral}Let $\alpha $ $\in \mathcal{O}_{K}$ (not
necessarily prime) and suppose $\alpha $ prime to $p$ in $\mathcal{O}_{K}$.
Then there exists a $k\in 
\mathbb{Z}
$, unique (modulo $p$), such that $\zeta _{p}^{k}\alpha $ is primary.
\end{proposition}

\begin{proof}
Consider the ideal $P=(1-\zeta _{p})$ in $\mathcal{O}_{K}$. Then the norm of
the ideal $N(P)=\prod\limits_{i=1}^{p-1}(1-\zeta _{p}^{i})=p$ by the fact
that $Gal(K:%
\mathbb{Q}
)\simeq U(%
\mathbb{Z}
/p%
\mathbb{Z}
)$. So $P$ is a prime ideal and is thus of degree $1$. So by Dedekind's
Theorem in Algebraic Number Theory, any element of $\mathcal{O}_{K}$ is the
root of a monic polynomial of degree $1$ in $\mathcal{O}_{K}/P$. So in the
particular case of $\alpha $, $\alpha -a_{0}=\overline{0}$ in $\mathcal{O}%
_{K}/P$ for some $a_{0}\in 
\mathbb{Z}
$. In other words, $\alpha \equiv a_{0}$ $(1-\zeta _{p})$. So $\frac{\alpha
-a_{0}}{(1-\zeta _{p})}\in \mathcal{O}_{K}$ and so, by the same argument, $%
\frac{\alpha -a_{0}}{(1-\zeta _{p})}\equiv a_{1}$ $(1-\zeta _{p})$ for some $%
a_{1}\in 
\mathbb{Z}
$. We stop repeating this here because multiplying the congruence by $%
(1-\zeta _{p})$, we now have a congruence modulo $(1-\zeta _{p})^{2}$, which
is what we want to consider. More precisely, we now have $\alpha
-a_{0}\equiv a_{1}(1-\zeta _{p})$ $\ (1-\zeta _{p})^{2}$, so $\alpha \equiv
a_{0}+a_{1}(1-\zeta _{p})$ $\ (1-\zeta _{p})^{2}$.

We want to eliminate the $(1-\zeta _{p})$ term by multiplying both sides by $%
\zeta _{p}^{n}$ for some $n\in 
\mathbb{Z}
$. Notice that $\zeta _{p}=(1-(1-\zeta _{p}))$. So modulo $(1-\zeta _{p})^{2}
$,%
\begin{eqnarray}
\zeta _{p}^{n}\alpha  &\equiv &\zeta _{p}^{n}a_{0}+a_{1}\zeta
_{p}^{n}(1-\zeta _{p})  \notag \\
&\equiv &a_{0}(1-(1-\zeta _{p}))^{n}+a_{1}(1-\zeta _{p})(1-(1-\zeta
_{p}))^{n}  \notag \\
&\equiv &a_{0}(1-n(1-\zeta _{p}))+a_{1}(1-\zeta _{p})(1-n(1-\zeta _{p}))%
\text{ }  \notag \\
&&  \notag
\end{eqnarray}

since considering $(1-(1-\zeta _{p}))^{n}$ as a polynomial in $(1-\zeta _{p})
$, $(1-\zeta _{p})^{2}$ divides $(1-\zeta _{p})^{i}$ for $i\geq 2$. So%
\begin{equation*}
\zeta _{p}^{n}\alpha \equiv a_{0}+(a_{1}-na_{0})(1-\zeta _{p})\text{ \ \ }%
(1-\zeta _{p})^{2}
\end{equation*}

Now $\alpha $ prime to $p$, so if $a_{0}\equiv 0$ $(p)$, then $a_{0}\equiv 0$
$(1-\zeta _{p})$, and so $\alpha \equiv 0$ $(1-\zeta _{p})$, which is a
contradiction. So $a_{0}\not\equiv 0$ $(p)$, and so $a_{1}-na_{0}\equiv 0$
has a \textbf{unique} solution $k$ modulo $p$. Now $(1-\zeta _{p})\mid
(1-\zeta _{p}^{2})$, and $N(\frac{1-\zeta _{p}^{2}}{1-\zeta _{p}})=\frac{%
N(1-\zeta _{p}^{2})}{N(1-\zeta _{p})}=1$, so $(1-\zeta _{p}^{2})$ is
associate to $(1-\zeta _{p})$. It follows that $(1-\zeta _{p})^{2}\mid p$,
and so $k$ is (still, since $a_{1}-na_{0}\in 
\mathbb{Z}
$) the \textbf{unique} integral solution modulo $p$ to $a_{1}-na_{0}\equiv 0$
\ $(1-\zeta _{p})^{2}$. Then $\zeta _{p}^{k}\alpha \equiv a_{0}$ $(1-\zeta
_{p})^{2}$, and therefore $\zeta _{p}^{k}\alpha $ is primary.
\end{proof}

\begin{lemma}
\label{ratioUnit}Let $u$ be a unit in $\mathcal{O}_{K}$. Then $\frac{u}{%
\overline{u}}=\zeta ^{t}$ for some $t\in 
\mathbb{Z}
$
\end{lemma}

\begin{proof}
Write $\upsilon =\frac{u}{\overline{u}}$. Conjugation is a Galois
automorphism on $\mathcal{O}_{K}$ since $\overline{\zeta }=\zeta ^{-1}=\zeta
^{p-1}$. So $\overline{u}$ is also a unit, and so $\upsilon \in \mathcal{O}%
_{K}$. Now let $\sigma _{k}$ be the $(p-1)$ Galois automorphisms on $%
\mathcal{O}_{K}$ such that $\sigma _{k}(\zeta )=\zeta ^{k}$, $k\in 
\mathbb{Z}
$. Then for all $1\leq k\leq (p-1)$, $\sigma _{k}\upsilon =\frac{\sigma _{k}u%
}{\sigma _{k}\overline{u}}=$ $\frac{\sigma _{k}u}{\overline{\sigma _{k}u}}$
by the above remark. So $\left\vert \sigma _{k}\upsilon \right\vert =\sigma
_{k}\upsilon \overline{\sigma _{k}\upsilon }=1$. So $\left\vert \sigma
_{k}\upsilon \right\vert ^{n}=1$ for any\thinspace $n\in 
\mathbb{N}
$.

Now consider the polynomial $f(x)=\prod\limits_{k=1}^{p-1}(x-\sigma
_{k}\upsilon )$. The coefficients of this polynomial are elementary
symmetric polynomials in $\{\sigma _{k}\upsilon :1\leq k\leq p-1\}$, and so
are invariant by action by $Gal(K:Q)=\{\sigma _{k}\upsilon :1\leq k\leq
p-1\} $. So $f(x)\in 
\mathbb{Z}
\lbrack x]$. But then the coefficient of $x^{k}$ is $s_{(p-1)-k}$ where $%
s_{j}$ is the $j^{th}$ elementary symmetric polynomial. But by the previous
paragraph, $\left\vert s_{(p-1)-k}\right\vert \leq
\sum\limits_{j=1}^{p-1-k}\left\vert \sigma _{k}\upsilon \right\vert ^{k}\leq
p-1-k$. So there are finitely many possible such $f(x)\in 
\mathbb{Z}
\lbrack x]$ since the coefficients are bounded. So there are finitely many
possible roots since a polynomial of finite degree has a finite number of
roots. But $\left\vert \sigma _{k}\upsilon ^{n}\right\vert =1$ for
any\thinspace $n\in 
\mathbb{N}
$, so $\{\upsilon ^{n}:n\in 
\mathbb{N}
\}$ satisfy the same argument. So we must have $\upsilon ^{n}=\upsilon
^{n^{\prime }}$ for some $n,n^{\prime }\in 
\mathbb{Z}
$. So $\upsilon ^{n-n^{\prime }}=1$, and it follows that $\upsilon $ is a
root of unity in $\mathcal{O}_{K}$.

So by the basic properties of prime cyclotomic fields, we must have $%
\upsilon =\pm \zeta ^{t}$ for some $t\in 
\mathbb{Z}
$. Now consider congruence modulo $\lambda =1-\zeta $. Then since $\dfrac{%
1-\zeta ^{k}}{1-\zeta }=\sum\limits_{i=1}^{k-1}\zeta ^{i}\in \left\langle
1,\zeta _{p},...,\zeta _{p}^{p-2}\right\rangle _{%
\mathbb{Z}
}=\mathcal{O}_{K}$, $\zeta ^{k}\equiv 1$ $(\lambda )$ for all $k\in 
\mathbb{Z}
$. So since $\overline{\zeta ^{k}}=\zeta ^{-k}\equiv 1\equiv \zeta ^{k}$ $%
(\lambda )$, $\alpha \equiv \overline{\alpha }$ $(\lambda )$ for all $\alpha
\in \mathcal{O}_{K}$. Namely, $u\equiv \overline{u}=\pm \zeta ^{-t}u\equiv
\pm u$ $(\lambda )$. So if $\upsilon =-\zeta ^{t}$, $u\equiv -u$ $(\lambda
)\Rightarrow 2u\equiv 0$ $(\lambda )$ which is impossible since $N(\lambda
)=p\nmid N(2u)=2^{p-1}$ since $p$ is odd. So $\upsilon =+\zeta ^{t}$.
\end{proof}

\begin{theorem}
\label{egreg}Let $u$ be a unit in $\mathcal{O}_{K}$. Then $u$ is real $%
\Leftrightarrow $ $u$ is primary in $\mathcal{O}_{K}$.
\end{theorem}

\begin{proof}
Since $\mathcal{O}_{K}=%
\mathbb{Z}
\lbrack \zeta _{p}]=\left\langle 1,\zeta _{p},...,\zeta
_{p}^{p-2}\right\rangle _{%
\mathbb{Z}
}$, we can write $u$ as $\sum\limits_{k=0}^{p-2}a_{k}\zeta ^{k}$ for unique $%
a_{0},...,a_{p-2}\in 
\mathbb{Z}
$. And so, noting that $\zeta ^{p-1}=-\sum\limits_{i=0}^{p-2}\zeta ^{i}$, $%
\zeta ^{-t}u=\sum\limits_{k=0}^{p-2}a_{k}\zeta
^{k-t}=\sum\limits_{k=0}^{p-2}(a_{k+t}-a_{(p-1)+t})\zeta ^{k}$ where $a_{k}$
is defined to be $a_{(k\text{ }\func{mod}p)}$ for all $k\notin \{0,...,p-1\}$
($a_{p-1}=0$, trivially). And so $\sum\limits_{k=0}^{p-2}(a_{p-k}-a_{1})%
\zeta ^{k}=\overline{u}=$ $\zeta
^{-t}u=\sum\limits_{k=0}^{p-2}(a_{k+t}-a_{(p-1)+t})\zeta ^{k}$ by \ref%
{ratioUnit} and therefore, since this representation is unique, we get%
\begin{equation}
a_{k+t}-a_{(p-1)+t}=a_{p-k}-a_{1}\text{ for all }0\leq k\leq p-1
\label{mainprimeq}
\end{equation}

Letting $k_{0}$ be the $\func{mod}p$ solution to $k+t\equiv p-k$ $(p)$, we
get $a_{k_{0}+t}=a_{p-k_{0}}$ and so (\ref{mainprimeq}) yields $%
a_{(p-1)+t}=a_{1}$. (\ref{mainprimeq}) then becomes%
\begin{equation}
a_{k+t}=a_{p-k}=a_{-k}\text{ for all }0\leq k\leq p-1  \label{mainprimeq1}
\end{equation}

Since replacing $k$ by $-(k+t)$ in (\ref{mainprimeq1}) leaves the equation
invariant, we get $\frac{p-1}{2}$ pairs of equal terms with distinct indices
amongst $a_{0},...,a_{p-1}$ (the 'remaining' term being $a_{k_{0}+t}$). Let $%
b_{1},...,b_{\frac{p-1}{2}}$ be representatives of these distinct pairs, and
let $b_{k_{0}+t}=a_{k_{0}+t}$ (we have simply selected and reordered the $%
a_{i}$'s).

Now by the proof of \ref{Only1PrimaryGeneral}, there is a unique $c$ modulo $%
p$ such that $\zeta ^{c}u$ is primary, and this $c$ is the solution to $%
ax\equiv b$ $(p)$ where $u\equiv a+b\lambda $ $(\lambda ^{2})$ where $%
\lambda =(1-\zeta )$. Now $u=$ $\sum\limits_{k=0}^{p-2}a_{k}\zeta ^{k}$.
Writing, as a polynomial, $f(x)=\sum\limits_{k=0}^{p-2}a_{k}x^{k}$, we can
find $a$ and $b$ by finding the coefficients of $1$ and $x$ respectively of $%
f(1-x)$ since $\zeta =1-\lambda $. Making elementary use of the Binomial
Theorem, we see that $f(1-x)=\sum\limits_{k=0}^{p-2}a_{k}(1-x)^{k}=\sum%
\limits_{k=0}^{p-2}a_{k}-\sum\limits_{k=0}^{p-2}ka_{k}x+...$ (we only need
the first two terms). So $c$ is the solution to%
\begin{equation}
\left( \sum\limits_{k=0}^{p-2}a_{k}\right) x\equiv
-\sum\limits_{k=0}^{p-2}ka_{k}\text{ }(p)  \label{mainprimeq21}
\end{equation}

Which, since $a_{p-1}=0$, is equivalent to%
\begin{equation}
\left( \sum\limits_{k=0}^{p-1}a_{k}\right) x\equiv
-\sum\limits_{k=0}^{p-1}ka_{k}\text{ }(p)  \label{mainprimeq3}
\end{equation}

Now $k_{0}+t\equiv p-k_{0}$ $(p)\Rightarrow k_{0}+t\equiv -(k_{0}+t)+t$ $%
(p)\Rightarrow (k_{0}+t)\equiv 2^{-1}t\Rightarrow
b_{k_{0}+t}=a_{k_{0}+t}=a_{2^{-1}t}$. Finally, note that for $%
a_{i}=a_{t-i}=b_{l}$ for $1\leq l\leq \frac{p-1}{2}$ by (\ref{mainprimeq1}), 
$ia_{i}+(t-i)a_{t-i}=tb_{l}$.

(\ref{mainprimeq3}) then becomes $\left( b_{k_{0}+t}+2\sum\limits_{k=1}^{%
\frac{p-1}{2}}b_{k}\right) x\equiv -\left( (2^{-1}t\func{mod}%
p)b_{k_{0}}+\sum\limits_{k=1}^{\frac{p-2}{2}}tb_{k}\right) $ $(p)$. It is
clear that $c\equiv -2^{-1}t$ $(p)$ is the solution to this congruence. By
its uniqueness, we see that $u$ is primary $\Leftrightarrow t\equiv 0$ $%
(p)\Leftrightarrow u=\zeta ^{t}\overline{u}$ is real.
\end{proof}

\section{Application to a Special Case of Fermat's Last Theorem}

Fermat's well-known final theorem, proved by Andrew Wiles and Richard Taylor
in 1994, states that%
\begin{equation*}
x^{n}+y^{n}=z^{n}
\end{equation*}

where $x,y,z,n\in 
\mathbb{Z}
$ has no non-trivial solutions $(x,y,z)$ for $n\geq 3$.

In fact, to prove this theorem, it suffices to prove that $x^{p}+y^{p}=z^{p}$
has no integral solutions for any positive odd prime $p$, since $%
x_{0}^{n}+y_{0}^{n}=z_{0}^{n}\Rightarrow x_{1}^{p}+y_{1}^{p}=z_{1}^{p}$
where $p$ is an odd prime dividing $n$ (exists since $n\geq 3$) and $%
(x_{1},y_{1},z_{1})=(x_{0}^{n/p},y_{0}^{n/p},z_{0}^{n/p})$. In other words,
we can restrict our study to the case where $n$ is an odd prime.

There is a very elegant proof of a special case of this theorem using
cyclotomy. The main use of the concept here is that it allows us to
transform a "sum of $n$-th powers" problem into a "divisibility" problem
since we can now factor $x^{p}+y^{p}$ as $\prod\limits_{i=0}^{p-1}(x+\zeta
_{p}^{i}y)$.\bigskip

In this section, we shall lay out said proof. Let $K=%
\mathbb{Q}
(\zeta )$ where $\zeta =\zeta _{p}$. We will suppose that for some $%
(x_{0,}y_{0},z_{0})$ is a solution to $x^{p}+y^{p}=z^{p}$ for some odd prime 
$p$. Then%
\begin{equation}
x_{0}^{p}+y_{0}^{p}=z_{0}^{p}  \label{FermatEq}
\end{equation}

WLOG, we can take $x_{0}$, $y_{0}$ and $z_{0}$ to be pairwise relatively
prime, for if some $d\in 
\mathbb{Z}
$ divides two of them, it must divide the 3$^{\text{rd}}$, and then $%
x_{0}^{p}+y_{0}^{p}=z_{0}^{p}\Leftrightarrow x_{1}^{p}+y_{1}^{p}=z_{1}^{p}$
where $x_{0},y_{0},z_{0}=dx_{1},dy_{1}$\bigskip $,dz_{1}$ respectively, with 
$x_{1},y_{1},z_{1}\in 
\mathbb{Z}
$.

We shall now reduce the problem to a special case and suppose that $p$\emph{%
\ does not divide the class number }$h$\emph{\ of }$O_{K}$, and that $p\nmid
x_{0}y_{0}z_{0}$. From (\ref{FermatEq}), we shall reach a
contradiction.\bigskip\ This case has been treated in Number Theory
textbooks such as \cite{BnS}. However, using the equivalence of primary and
real units in $\mathcal{O}_{K}$ when $K$ is a prime cyclotomic field, we can
prove the result more rapidly.

\begin{lemma}
\label{CoprimeIdeals}Let $i\not\equiv j$ $(p)$. Then the ideals $%
I=(x_{0}+\zeta ^{i}y_{0})$ and $J=(x_{0}+\zeta ^{j}y_{0})$ are relatively
prime.
\end{lemma}

\begin{proof}
Consider the ideal $I+J$. $J$ contains the element $-(x_{0}+\zeta ^{j}y_{0})$%
, so $x_{0}+\zeta ^{i}y_{0}-(x_{0}+\zeta ^{j}y_{0})=(\zeta ^{i}-\zeta
^{j})y_{0}\in I+J$. Likewise, since $\mathcal{O}_{K}=%
\mathbb{Z}
\lbrack \zeta ]$, $-\zeta ^{j}(x_{0}+\zeta ^{i}y_{0})=\zeta ^{j}x_{0}+\zeta
^{i+j}y_{0}\in I$ and $\zeta ^{i}(x_{0}+\zeta ^{j}y_{0})=\zeta
^{i}x_{0}+\zeta ^{i+j}y_{0}\in J$. So $\zeta ^{i}x_{0}+\zeta
^{i+j}y_{0}-\zeta ^{j}(x_{0}+\zeta ^{i}y_{0})=(\zeta ^{i}-\zeta
^{j})x_{0}\in I+J$. Now $(x_{0},y_{0})=1\Rightarrow $ there exist $a,b\in 
\mathbb{Z}
$ such that $ax_{0}+by_{0}=1$. So $a(\zeta ^{i}-\zeta ^{j})x_{0}+b(\zeta
^{i}-\zeta ^{j})y_{0}=(\zeta ^{i}-\zeta ^{j})\in I+J$. 

Now $N(\zeta ^{i}-\zeta ^{j})=p$ since $(N(\zeta ^{i}-\zeta ^{j}))^{2}=$ $%
\prod\limits_{k=1}^{p-1}(\zeta ^{ik}-\zeta
^{jk})^{2}=\prod\limits_{k=1}^{p-1}(-\zeta ^{-k(j-i)})(1-\zeta
^{k(j-i)})^{2}=\prod\limits_{k=1}^{p-1}(-\zeta ^{-k})(1-\zeta
^{k})^{2}=+\zeta ^{-p\frac{p-1}{2}}\prod\limits_{k=1}^{p-1}(1-\zeta
^{k})^{2}=1\cdot \left( \sum\limits_{k=1}^{p-1}1\right) ^{2}=p^{2}$. So $%
N(I+J)\mid p$. If $N(I+J)=p$, then since $I\subseteq I+J$, $p=N(I+J)\mid
N(I)=\prod\limits_{i=0}^{p-1}(x_{0}+\zeta
^{i}y_{0})=x_{0}^{p}+y_{0}^{p}=z_{0}^{p}$. So since $p$ is prime, $p\mid
z_{0}\Rightarrow $ contradiction. So $N(I+J)=1$, and therefore $I+J=\mathcal{%
O}_{K}$. So $I$ and $J$ are coprime since $P\mid I$ and $P\mid J\Rightarrow
P\mid I+J\Rightarrow P=\mathcal{O}_{K}$.
\end{proof}

Now $x_{0}^{p}+y_{0}^{p}=z_{0}^{p}\Rightarrow
\prod\limits_{i=0}^{p-1}(x_{0}+\zeta ^{i}y_{0})=(z_{0})^{p}$ as ideals. But $%
\{(x_{0}+\zeta ^{i}y_{0}):0\leq i\leq p-1\}$ are pairwise coprime. So by
unique factorization of ideals, each of these ideals must be a $p$-th power.
So in particular, taking $i=1$, $(x_{0}+\zeta y_{0})=\mathfrak{I}^{p}$ for
some ideal $\mathfrak{I}$. So since $(x_{0}+\zeta y_{0})$ is principal, $[%
\mathfrak{I]}$ has order dividing $p$ in the ideal class group, but since $%
p\nmid h$, we must have that the order of $[\mathfrak{I]}$ is $1$.\ So $%
\mathfrak{I}$ is principal. Let $\mathfrak{I}=(\alpha )$. Then $(x_{0}+\zeta
y_{0})=(\alpha ^{p})$, and so $x_{0}+\zeta y_{0}$ is associate to $\alpha
^{p}$. We write $x_{0}+\zeta y_{0}=u\alpha ^{p}$ where $u$ is a unit in $%
\mathcal{O}_{K}$.

Then by \ref{Only1PrimaryGeneral} there exists a unique $c$ modulo $p$ such
that $\zeta ^{-c}u$ is primary. Let $\zeta ^{-c}u=u_{0}$ so that $u=\zeta
^{c}u_{0}$ where $u_{0}$ is primary. But $u_{0}$ is trivially a unit, and is
therefore real by \ref{egreg}.

So $x_{0}+\zeta y_{0}=\zeta ^{c}u_{0}\alpha ^{p}$ where $u_{0}$ is real.
Note that modulo $p$, $\alpha ^{p}\equiv \left(
\sum\limits_{i=0}^{p-2}a_{i}\zeta ^{i}\right) ^{p}\equiv
\sum\limits_{i=0}^{p-2}a_{i}^{p}\zeta ^{ip}\equiv
\sum\limits_{i=0}^{p-2}a_{i}^{p}\in 
\mathbb{Z}
$ \ $(p)$. So $\alpha ^{p}\equiv \overline{\alpha ^{p}}$ $(p)$. It follows
that $x_{0}+\zeta y_{0}=\zeta ^{c}u_{0}\alpha ^{p}\Rightarrow x_{0}+\zeta
y_{0}\equiv \zeta ^{c}u_{0}\alpha ^{p}$ $(p)\Rightarrow \overline{%
x_{0}+\zeta y_{0}}\equiv \overline{\zeta ^{c}u_{0}\alpha ^{p}}$ $%
(p)\Rightarrow x_{0}+\zeta ^{-1}y_{0}\equiv \zeta ^{-c}u_{0}\alpha ^{p}$ $%
(p) $. So we now have $x_{0}+\zeta y_{0}\equiv \zeta ^{c}u_{0}\alpha ^{p}$ $%
(p)\Rightarrow \zeta ^{-c}x_{0}+\zeta ^{1-c}y_{0}\equiv u_{0}\alpha ^{p}$ $%
(p)$ and $x_{0}+\zeta ^{-1}y_{0}\equiv \zeta ^{-c}u_{0}\alpha ^{p}$ $%
(p)\Rightarrow \zeta ^{c}x_{0}+\zeta ^{c-1}y_{0}\equiv u_{0}\alpha ^{p}$ $%
(p) $. Subtracting the latter congruence from the former yields%
\begin{equation}
\zeta ^{-c}x_{0}+\zeta ^{1-c}y_{0}-\zeta ^{c}x_{0}-\zeta ^{c-1}y_{0}\equiv 0%
\text{ }(p)  \label{congeq1}
\end{equation}

Now an element of $\mathcal{O}_{K}=%
\mathbb{Z}
\lbrack \zeta ]$ is divisible by $p$ if and only if all of the coefficients
as a polynomial in $\zeta $ are divisible by $p$. $p\nmid x_{0},y_{0}$ since 
$p\nmid x_{0}y_{0}z_{0}$, so we must check the cases where one of $%
\{c,-c,1-c,c-1\}$ is congruent to $-1$ modulo $p$ or where two of $%
\{c,-c,1-c,c-1\}$ are equal modulo $p$. These cases can be split as follows:

\begin{itemize}
\item $c\equiv 0$ $(p)$ (so that $c\equiv -c$ $(p)$). Then $p\mid
y_{0}(\zeta -\zeta ^{-1})=y_{0}(\sum\limits_{i=2}^{p-2}\zeta
^{i}+1)\Rightarrow p\mid y_{0}$ (even if $p=3$) $\Rightarrow $ contradiction.

\item $c\equiv 1$ $(p)$ (so that $1-c\equiv c-1$ $(p)$). Then $p\mid
x_{0}(\zeta ^{-1}-\zeta )\Rightarrow p\mid x_{0}$ as in the previous case $%
\Rightarrow $ contradiction.

\item $c\equiv 2^{-1}$ $(p)$ (so that $c\equiv 1-c$ $(p)$). Then $p\mid
(y_{0}-x_{0})\zeta ^{c}+\zeta ^{-c}(x_{0}-y_{0})$. So $p\mid (x_{0}-y_{0})$.
We then rewrite \ref{FermatEq} as $x_{0}^{p}+(-z_{0})^{p}=(-y_{0})^{p}$
(since $p$ is odd). Then with the same argument we will get $p\mid
(x_{0}+z_{0})$. But \ref{FermatEq} yields $x_{0}^{p}+y_{0}^{p}-z_{0}^{p}%
\equiv 0$ $(p)$ and so $x_{0}+y_{0}-z_{0}\equiv 0$ $(p)$. This yields $%
3x_{0}\equiv 0$ $(p)$. We suppose for now that $p>3$. Then this yields $%
p\mid x_{0}\Rightarrow $ contradiction.

\item Letting one of $\{c,-c,1-c,c-1\}$ be congruent to $-1$ modulo $p$ will
yield one of the coefficients of the terms of (\ref{congeq1}) as $\pm
(x_{0}-y_{0})$, giving the same contradiction as in the previous case.
\end{itemize}

\bigskip 

We therefore obtain a contradiction in all cases. We have, however, supposed
that $p>3\,$. A general study of the case where $p=3$ is done elegantly in 
\cite{Flynn}.\bigskip 

\section{An Approach to Pell's Equation using cyclotomy}

Pell's Equation is%
\begin{equation*}
x^{2}-dy^{2}=1\text{, \ \ }x,y\in 
\mathbb{Z}%
\end{equation*}

in $x$ and $y$, where $d\in 
\mathbb{Z}
^{+}$. $d\leq 0$ trivially yields the single solution $(1,0)$, and we can
consider $d$ to be square-free, since any square factor of $d$ can be
incorporated into $y$.

The equation can be solved using cyclotomy and quadratic residues. A partial
solution was found by Dirichlet using this method, building upon the work of
Gauss \cite{Dir}. In this section, we build upon Dirichlet's work,
explicitly writing the solution and using the modern machinery of Galois
Theory to streamline the approach. Again, we let $p$ be an odd prime, and
define $p^{\ast }=(-1)^{\frac{p-1}{2}}p$, $i=\sqrt{-1}$, and start by
introducing an important lemma.

\begin{lemma}
\label{MainPellLemma}$\left\{ 
\begin{array}{l}
q_{1}(x)=2\prod\limits_{\substack{ 1\leq k<p  \\ \QOVERD( ) {k}{p}=1}}%
(x-\zeta ^{k})=f(x)+\sqrt{p^{\ast }}g(x) \\ 
q_{-1}(x)=2\prod\limits_{\substack{ 1\leq k<p  \\ \QOVERD( ) {k}{p}=-1}}%
(x-\zeta ^{k})=f(x)-\sqrt{p^{\ast }}g(x)%
\end{array}%
\right. $ where $f(x),g(x)$ are polynomials in $%
\mathbb{Z}
\lbrack x]$.
\end{lemma}

\begin{proof}
Note that the product of the 2 above polynomials (on the left-hand side) is $%
4\prod\limits_{1\leq k<p}(x-\zeta ^{k})=4m_{p}(x)\in 
\mathbb{Z}
\lbrack x]$. It is therefore fixed by any Galois automorphism in $Gal(K:%
\mathbb{Q}
)$. Now taking $\theta =\zeta ^{\frac{p^{2}-1}{8}}\prod\limits_{k=1}^{\frac{%
p-1}{2}}(1-\zeta ^{k})^{2}$, we see that $\theta ^{2}=p^{\ast }$ since $%
(-1)^{\frac{p^{2}-1}{8}}\equiv \left( \frac{2}{p}\right) $ $($mod $2)$, and
trivially $\theta \in \mathcal{O}_{K}$. So $\sqrt{p^{\ast }}\in \mathcal{O}%
_{K}$, Now an automorphism $\sigma $ in the Galois group fixes $p^{\ast }$
if and only if $\sigma $ is a square. But this is if and only if $\sigma $
fixes all (and only) the $\zeta ^{k}$ such that $k$ is a quadratic residue
modulo $p$. So $\prod\limits_{\substack{ 1\leq k<p \\ \QOVERD( ) {k}{p}=1}}%
(x-\zeta ^{k})\in L[x]$ where $L=%
\mathbb{Q}
(\sqrt{p^{\ast }})$. All the coefficients in $L[x]$ are of the form $a+b%
\sqrt{p^{\ast }}$ where $a$ and $b$ are both rational, and $\frac{1}{2}\cdot 
$ an algebraic integer (allowing for the fact that $p^{\ast }\equiv 1$ $(4)$%
). The coefficients of $2\prod\limits_{\substack{ 1\leq k<p \\ \QOVERD( )
{k}{p}=1}}(x-\zeta ^{k})$ are therefore rational algebraic integers and thus
in $%
\mathbb{Z}
$. We can now expand $q_{1}(x)$ and rewrite it as $q_{1}(x)=f(x)+\sqrt{%
p^{\ast }}g(x)$ where $f(x),g(x)$ are polynomials in $%
\mathbb{Z}
\lbrack x]$.{}

A similar argument shows that $q_{-1}(x)\in L[x]$. Now let $\tau $ be the
Galois automorphism in $Gal(K:Q)$ defined by $\tau (\sqrt{p^{\ast }})=-\sqrt{%
p^{\ast }}$ (noting that $K:L:%
\mathbb{Q}
$ is a tower of fields). Then by the above, and since $\tau ^{2}$ must fix $%
q_{1}(x)$, we must have that $\tau (\zeta ^{k})=\zeta ^{l}$ where $\QOVERD(
) {k}{p}\QOVERD( ) {l}{p}=-1$. So since $\tau $ is a Galois automorphism
over $K$, we must have $\tau (q_{1}(x))=q_{-1}(x)$. This yields that $%
q_{-1}(x)=f(x)-\sqrt{p^{\ast }}g(x)$.
\end{proof}

We will primarily consider the case where $d$ is an odd prime. Pell's
Equation then becomes%
\begin{equation}
x^{2}-py^{2}=1  \label{Pell's}
\end{equation}

By Lemma \ref{MainPellLemma},%
\begin{equation*}
4m_{p}(x)=q_{1}(x)q_{-1}(x)=f(x)^{2}-(p^{\ast })g(x)^{2}
\end{equation*}

And so, replacing $x$ by $1$, we get%
\begin{equation}
4p=x_{1}^{2}-p^{\ast }y_{1}^{2}\text{ where }x_{1}=f(1)\text{, }y_{1}=g(1)
\label{Dirichlet1}
\end{equation}

Since $f(x),g(x)\in 
\mathbb{Z}
\lbrack x]$, $x_{1},y_{1}\in 
\mathbb{Z}
$, and we can see that Lemma \ref{MainPellLemma} relates to Pell's Equation
insofar as it gives us a pair $(x_{1},y_{1})$ that verifies an equation very
similar to (\ref{Pell's}).

$4p=x_{1}^{2}-p^{\ast }y_{1}^{2}\Rightarrow x_{1}^{2}=4p+p^{\ast
}y_{1}^{2}\Rightarrow p\mid x_{1}^{2}\Rightarrow p\mid x_{1}$ since $p$ is
prime. So letting $p\xi _{1}=x_{1}$, we can rewrite equation (\ref%
{Dirichlet1}) as $4p=p^{2}\xi _{1}^{2}-p^{\ast }y_{1}^{2}$, and so, dividing
by $p$,%
\begin{equation}
p\xi _{1}^{2}-(-1)^{\frac{p-1}{2}}y_{1}^{2}=4  \label{Dirichlet2}
\end{equation}

We now analyze $q_{1}(x)$ and $q_{-1}(x)$ to obtain some insight as to the
values $x_{1}$ and $y_{1}$. $x^{2}\equiv (p-x)^{2}$ $(p)$, so all quadratic
residues are in $\{x^{2}$ $(p):1\leq x\leq \frac{p-1}{2}\}$. We can
therefore reorder the terms in $q_{1}(x)$ and write it as $%
2\prod\limits_{k=1}^{\frac{p-1}{2}}(x-\zeta ^{k^{2}})$, and so $%
q_{1}(1)=2\prod\limits_{k=1}^{\frac{p-1}{2}}(1-\zeta ^{k^{2}})$.

The value of $p^{\ast }$ depends on the value of $p$ modulo $4$ so we will
consider the two cases separately for simplicity.

\bigskip

\textbf{Case 1:}\emph{\ }$p\equiv 1$ $(4)$.

Then (\ref{Dirichlet2}) becomes $p\xi _{1}^{2}-y_{1}^{2}=4$ (or, to
emphasize the similarity to Pell's Equation, $y_{1}^{2}-p\xi _{1}^{2}=-4$).

We then have two subcases.

If $p\equiv 1$ $(8)$, then $y_{1}^{2}-\xi _{1}^{2}\equiv 4$ $(8)$. Trivially 
$y_{1}$ and $\xi _{1}$ must either be both odd or both even. But $%
1^{2}\equiv 3^{2}\equiv 5^{2}\equiv 7^{2}\equiv 1$ $(8)$, so if $y_{1}$ and $%
\xi _{1}$ were both odd we would have $y_{1}^{2}-\xi _{1}^{2}\equiv 0$ $%
(8)\Rightarrow $ contradiction. It follows that $y_{1}$ and $\xi _{1}$ are
both even, and we can thus write $y_{2}=\frac{y_{1}}{2},\xi _{2}=\frac{\xi
_{1}}{2}\in 
\mathbb{Z}
$. Then $y_{2}-p\xi _{2}^{2}=-1$. We can use the fact that $(\sqrt{p}%
)^{2}\in 
\mathbb{Z}
$ to get rid of the minus sign in front of $1$. $y_{2}^{2}-p\xi _{2}^{2}=-1$
yields $(y_{2}-\sqrt{p}\xi _{2})(y_{2}+\sqrt{p}\xi _{2})=-1$, and so $(y_{2}-%
\sqrt{p}\xi _{2})^{2}(y_{2}+\sqrt{p}\xi _{2})^{2}=1$. But $(y_{2}\pm \sqrt{p}%
\xi _{2})^{2}=a\pm b\sqrt{p}$, $a,b\in 
\mathbb{Z}
$. Taking $(x,y)=(a,b)$, we have solved (\ref{Pell's}). Summarizing, we get
a solution from%
\begin{eqnarray*}
(a,b) &=&\left( \frac{1}{4}(g(1)^{2}+\frac{f(1)^{2}}{p})\text{ },\text{ }%
\frac{f(1)g(1)}{2p}\right) \text{ } \\
&&\text{where we can directly compute }f(1)\text{ and }g(1)
\end{eqnarray*}

If $p\equiv 5$ $(8)$, $y_{1}^{2}+3\xi _{1}^{2}\equiv 4$ $(8)$. Given that
the only quadratic residues modulo $8$ are $0,1,4$, we must have $%
(y_{1}^{2},\xi _{1}^{2})\equiv (1,1),(0,4)$ or $(4,0)$ $\ (8)$.

We now use the fact that $8^{2}=2^{2\cdot 3}=4^{3}$ and consider $(y_{1}+%
\sqrt{p}\xi _{1})^{3}=(y_{1}^{3}+3p\xi _{1}^{2}y_{1})+\sqrt{p}(p\xi
_{1}^{3}+3y_{1}^{2}\xi _{1})=y_{2}+\sqrt{p}\xi _{2}$ and see that $%
y_{2}^{2}-p\xi _{2}^{2}=(y_{1}^{2}-p\xi _{1}^{2})^{3}=-4^{3}$.

But $y_{2}=y_{1}(y_{1}^{2}+3p\xi _{1}^{2})\equiv y_{1}(y_{1}^{2}-\xi
_{1}^{2})$ $(8)$. $(y_{1}^{2},\xi _{1}^{2})\equiv (1,1)$ $(8)\Rightarrow $ $%
y_{2}\equiv 0$ $(8)$. $(y_{1}^{2},\xi _{1}^{2})\equiv (0,4)$ or $(4,0)$ $%
(8)\Rightarrow y_{2}\equiv 4\cdot 4,$ $0\cdot 4$ or $\pm 2\cdot 4\equiv 0$ $%
(8)$. So in any case $y_{2}\equiv 0$ $(8)$.

Similarly $\xi _{2}=\xi _{1}(p\xi _{1}^{2}+3y_{1}^{2})\equiv \xi _{1}(5\xi
_{1}^{2}+3y_{1}^{2})$ $(8)$. $(y_{1}^{2},\xi _{1}^{2})\equiv (1,1)$ $%
(8)\Rightarrow $ $\xi _{2}\equiv \xi _{2}(5+3)\equiv 0$ $(8)$. $%
(y_{1}^{2},\xi _{1}^{2})\equiv (0,4)$ or $(4,0)$ $(8)\Rightarrow \xi
_{2}\equiv \pm 2\cdot 4,$ $0\cdot 4$ or $4\cdot 0\equiv 0$ $(8)$. So in any
case $\xi _{2}\equiv 0$ $(8)$.

So $8\mid y_{2},\xi _{2}$ and thus, writing $y_{3}=\frac{y_{2}}{8},\xi _{3}=%
\frac{\xi _{2}}{8}\in 
\mathbb{Z}
$, we get $(y_{3}^{2}-p\xi _{3}^{2})=\frac{-4^{3}}{8^{2}}=-1$. As in the
case where $p\equiv 1$ $(8)$, writing $(y_{3}\pm \sqrt{p}\xi _{3})^{2}=a\pm b%
\sqrt{p}$, $a,b\in 
\mathbb{Z}
$, $(x,y)=(a,b)$ is a solution of (\ref{Pell's}). Summarizing, we get a
solution from%
\begin{equation*}
(a,b)=\left( 
\begin{array}{c}
\frac{1}{64}((g(1)^{3}+\frac{3f(1)^{2}g(1)}{p})^{2}+p(\frac{f(1)^{3}}{p^{2}}%
+3\frac{g(1)^{2}f(1)}{p})^{2})\text{ }, \\ 
\text{ }\frac{1}{32}(g(1)^{3}+3\frac{f(1)^{2}g(1)}{p})(\frac{f(1)^{3}}{p^{2}}%
+3\frac{g(1)^{2}f(1)}{p})%
\end{array}%
\right) 
\end{equation*}

\bigskip

\textbf{Case 2:} $p\equiv 3$ $(4)$.

Let $l=\frac{p-1}{2}$. $p\equiv 3$ $(4)\Rightarrow l$ is odd. We see that $%
f(x)=\frac{1}{2}(q_{1}(x)+q_{-1}(x))=\prod\limits_{\substack{ 1\leq k<p \\ %
\QOVERD( ) {k}{p}=1}}(x-\zeta ^{k})+\prod\limits_{\substack{ 1\leq k<p \\ %
\QOVERD( ) {k}{p}=-1}}(x-\zeta ^{k})$. $f$ is of degree $l$. We shall find a
relation amongst the coefficients of $f$ by comparing $f(\zeta )$ and $f(%
\overline{\zeta })=\overline{f(\zeta )}$ (since $f(x)\in 
\mathbb{Z}
\lbrack x]$). Trivially $\QOVERD( ) {1}{p}=1$, so $\prod\limits_{\substack{ %
1\leq k<p \\ \QOVERD( ) {k}{p}=1}}(\zeta -\zeta ^{k})=0$ and so $f(\zeta
)=\prod\limits_{\substack{ 1\leq k<p \\ \QOVERD( ) {k}{p}=-1}}(\zeta -\zeta
^{k})$. Also note that $\QOVERD( ) {-1}{p}=(-1)^{\frac{p-1}{2}}=-1$, and so $%
\QOVERD( ) {k}{p}=-\QOVERD( ) {-k}{p}$ for all $1\leq k\leq p-1$. So $%
f(\zeta )=\prod\limits_{\substack{ 1\leq k<p \\ \QOVERD( ) {k}{p}=1}}(\zeta
-\zeta ^{-k})$. By the same line of reasoning, $f(\overline{\zeta })=f(\zeta
^{-1})=$ $\prod\limits_{\substack{ 1\leq k<p \\ \QOVERD( ) {k}{p}=1}}(\zeta
^{-1}-\zeta ^{k})$. So%
\begin{eqnarray*}
\frac{f(\zeta )}{f(\zeta ^{-1})} &=&\prod\limits_{\substack{ 1\leq k<p \\ %
\QOVERD( ) {k}{p}=1}}\frac{(\zeta -\zeta ^{-k})}{(\zeta ^{-1}-\zeta ^{k})}%
=(-1)^{l}\prod\limits_{\substack{ 1\leq k<p \\ \QOVERD( ) {k}{p}=1}}\zeta
^{1-k} \\
&&\text{since there are precisely }l\text{ quadratic residues modulo }p \\
&=&-\zeta ^{l}\prod\limits_{\substack{ 1\leq k<p \\ \QOVERD( ) {k}{p}=1}}%
\zeta ^{-k} \\
&=&-\zeta ^{l} \\
&&\text{since }\sum\limits_{\substack{ 1\leq k<p \\ \QOVERD( ) {k}{p}=1}}k=p%
\frac{p-1}{2}+0\text{ since the Legendre symbol is a } \\
&&\text{quadratic character modulo }p\text{ and since }\left( \frac{0}{p}%
\right) =0\text{.}
\end{eqnarray*}

So $f(\zeta )=-\zeta ^{l}f(\zeta ^{-1})$. So writing $%
f(x)=a_{l}x^{l}+a_{l-1}x^{l-1}+...+a_{1}x+a_{0}$, this yields $a_{l}\zeta
^{l}+a_{l-1}\zeta ^{l-1}+...+a_{1}\zeta +a_{0}=-a_{0}\zeta ^{l}-a_{1}\zeta
^{l-1}-...-a_{l-1}\zeta -a_{l}$, i.e. 
\begin{equation}
\sum\limits_{k=0}^{l}a_{k}\zeta ^{k}=\sum\limits_{k=0}^{l}(-a_{k})\zeta
^{l-k}  \label{equi1'}
\end{equation}

Now it is trivial to see that $a_{l}=2$ by the above formula for $f(x)$.
Also, $q_{1}(x)=2\prod\limits_{\substack{ 1\leq k<p \\ \QOVERD( ) {k}{p}=1}}%
(x-\zeta ^{k})$. The constant term of $q_{1}$ is $2(-1)^{l}\prod\limits
_{\substack{ 1\leq k<p \\ \QOVERD( ) {k}{p}=1}}\zeta
^{k}=-2\prod\limits_{1\leq k\leq l}\zeta ^{k^{2}}=-2\zeta ^{\frac{%
l(l+1)(2l+1)}{6}}=-2\zeta ^{p\frac{p^{2}-1}{24}}$. Now $3\mid p^{2}-1$ since 
$p\neq 3$ ($p\equiv 3$ $(4)$), and $p^{2}\equiv 1$ $(8)$ since $p$ is odd.
So $3\cdot 8=24\mid p^{2}-1$. So The constant term of $q_{1}$ is $-2\cdot
1=-2$. But $q_{1}(x)=f(x)+\sqrt{p^{\ast }}g(x)$ where $f(x),g(x)\in 
\mathbb{Z}
\lbrack x]$. So we must have $a_{0}=-2$. Therefore $a_{l}=-a_{0}$. So (\ref%
{equi1'}) now yields $\sum\limits_{k=1}^{l-1}a_{k}\zeta
^{k}=\sum\limits_{k=1}^{l-1}(-a_{k})\zeta
^{l-k}=\sum\limits_{k=1}^{l-1}(-a_{l-k})\zeta ^{k}$ (after replacing $k$ by $%
l-k$), and $\{\zeta ,...,\zeta ^{l-1}\}$ is a $%
\mathbb{Z}
$-linearly independent subset. So $a_{l-k}=-a_{l}$ for $1\leq k\leq l-1$,
and so by the above, $a_{l-k}=-a_{l}$ for all $0\leq k\leq l$. We can
therefore rewrite $f(x)$ as $%
2(x^{l}-1)+b_{1}x(x^{l-2}-1)+b_{2}x^{2}(x^{l-4}-1)+...+b_{\frac{l-1}{2}}x^{%
\frac{l-1}{2}}(x-1)=\sum\limits_{k=0}^{\frac{l-1}{2}}b_{k}x^{k}(x^{l-2k}-1)$%
, $b_{k}\in 
\mathbb{Z}
$ for all $0\leq k\leq \frac{l-1}{2}$ (with $b_{0}=2$).

Replacing $x$ by $i=\sqrt{-1}$, we see that $x^{k}(x^{l-2k}-1)$ depends on
whether $p\equiv 3$ or $7$ $(8)$.

Let $p\equiv 3$ $(8)$. Then $l\equiv 1$ $(4)$ and simple calculation yields

$i^{k}(i^{l-2k}-1)=\left\{ 
\begin{array}{ll}
1-i & \text{if }k\equiv 1,2\text{ }(4) \\ 
-(1-i) & \text{if }k\equiv 0,3\text{ }(4)%
\end{array}%
\right. $

$p\equiv 7$ $(8)\Rightarrow l\equiv 3$ $(4)$, and the same type of
calculation yields

$i^{k}(i^{l-2k}-1)=\left\{ 
\begin{array}{ll}
1+i & \text{if }k\equiv 3,2\text{ }(4) \\ 
-(1+i) & \text{if }k\equiv 0,1\text{ }(4)%
\end{array}%
\right. $

Writing $i^{\ast }=\left\{ 
\begin{array}{ll}
-i & \text{if }p\equiv 3\text{ }(8) \\ 
+i & \text{if }p\equiv 7\text{ }(8)%
\end{array}%
\right. $, we see that $f(i)=\sum\limits_{k=0}^{\frac{l-1}{2}}\pm
b_{k}(1+i^{\ast })=y_{2}(1+i^{\ast })$ where $y_{2}\in 
\mathbb{Z}
$.

Now, 
\begin{eqnarray*}
g(x) &=&\frac{1}{2\sqrt{p^{\ast }}}(q_{1}(x)-q_{-1}(x)) \\
&=&\frac{1}{\sqrt{p^{\ast }}}\left( \prod\limits_{\substack{ 1\leq k<p  \\ %
\QOVERD( ) {k}{p}=1}}(x-\zeta ^{k})-\prod\limits_{\substack{ 1\leq k<p  \\ %
\QOVERD( ) {k}{p}=-1}}(x-\zeta ^{k})\right)
\end{eqnarray*}%
And so%
\begin{eqnarray*}
g(\zeta ) &=&\boldsymbol{-}\frac{1}{\sqrt{p^{\ast }}}\left( \prod\limits 
_{\substack{ 1\leq k<p  \\ \QOVERD( ) {k}{p}=1}}(\zeta -\zeta ^{-k})\right)
\\
\text{and }g(\zeta ^{-1}) &=&\frac{1}{\sqrt{p^{\ast }}}\left( \prod\limits 
_{\substack{ 1\leq k<p  \\ \QOVERD( ) {k}{p}=1}}(\zeta ^{-1}-\zeta
^{k})\right)
\end{eqnarray*}%
A similar line of reasoning as for $f(x)$ gives us that $g(\zeta )=+\zeta
^{l}g(\zeta ^{-1})$. Following the same steps as for $f(x)$, we find that,
writing $g(x)$ as $\frac{1}{\sqrt{p^{\ast }}}\sum\limits_{k=0}^{l}a_{k}x^{k}$%
, we get $a_{l-k}=+a_{l}$ for all $0\leq k\leq l$ (with $a_{l}=a_{0}=0$ this
time). We can therefore similarly rewrite $g(x)$ as $\sum\limits_{k=0}^{%
\frac{l-1}{2}}b_{k}x^{k}(x^{l-2k}+1)$, $b_{k}\in 
\mathbb{Z}
$ (remembering that $g(x)\in 
\mathbb{Z}
\lbrack x]$ by \ref{MainPellLemma}). A similar argument shows that $%
g(i)=\sum\limits_{k=0}^{\frac{l-1}{2}}\pm b_{k}(1-i^{\ast })=\xi
_{2}(1-i^{\ast })$ where $\xi _{2}\in 
\mathbb{Z}
$.

Now $l\equiv 3$ $(4)$, so $q_{1}(i)q_{-1}(i)=4m_{p}(i)=4(1+i+...+i^{l})=4%
\cdot ((1+i-1-i)+(1+i-1-i)+...+(1+i-1))=4i$

So $f(i)^{2}-p^{\ast }g(i)^{2}=f(i)^{2}+pg(i)^{2}=4i$, and so $%
y_{2}^{2}(1+i^{\ast })^{2}+p\xi _{2}(1-i^{\ast })^{2}=2y_{2}^{2}i^{\ast
}-2p\xi _{2}^{2}i^{\ast }=4i$ or, dividing by $2i^{\ast }=\pm 2i$,%
\begin{eqnarray*}
&&y_{2}^{2}-p\xi _{2}^{2}=\pm 2 \\
&\Rightarrow &(y_{2}+\sqrt{p}\xi _{2})^{2}(y_{2}-\sqrt{p}\xi _{2})^{2}=4
\end{eqnarray*}

Now $y_{2},\xi _{2}$ are odd, else $y_{2}^{2}-p\xi _{2}^{2}\equiv
y_{2}^{2}+\xi _{2}^{2}\equiv 0\not\equiv \pm 2$ $(4)$. So the coefficients
of $(y_{2}+\sqrt{p}\xi _{2})^{2}=(y_{2}^{2}+p\xi _{2}^{2})+2y_{2}\xi _{2}%
\sqrt{p}$ are even. We can thus write $a=\frac{(y_{2}^{2}+p\xi _{2}^{2})}{2}%
,b=y_{2}\xi _{2}\in 
\mathbb{Z}
$ and get%
\begin{equation*}
a^{2}-pb^{2}=\frac{(y_{2}+\sqrt{p}\xi _{2})^{2}(y_{2}-\sqrt{p}\xi _{2})^{2}}{%
2\cdot 2}=\frac{4}{4}=1
\end{equation*}

This solves the equation, where%
\begin{eqnarray*}
(a,b) &=&\left( \frac{i^{\ast }}{4}(pg(i)^{2}-f(i)^{2})\text{ },\text{ }%
\frac{1}{2}g(i)f(i)\right) \text{ } \\
&&\text{where we can directly compute }f(i)\text{ and }g(i)
\end{eqnarray*}

\bigskip

To apply this method to the general case of Pell's Equation (where $d$ is
square-free but not necessarily prime), since $d$ is square-free, it can be
written as $d=\prod\limits_{k=1}^{r}p_{k}$ where the $p_{k}$'s are rational
primes. So it suffices to study the case where $d=pq$ for primes $p$ and $q$
and deduce the general case by induction. We will not describe said case in
depth here since this paper mainly focuses on prime cyclotomic fields, but
we remark that taking $%
\mathbb{Q}
(\zeta _{pq})$, $m_{pq}(x)=m_{p}(x)m_{q}(x)\frac{(x^{pq}-1)/(x-1)}{%
((x^{p}-1)/(x-1))\cdot ((x^{q}-1)/(x-q))}=\frac{(x^{pq}-1)(x-1)}{%
(x^{p}-1)(x^{q}-1)}$ which can be shown to be irreducible by a similar
method as the simple proof for showing that $\sum\limits_{k=0}^{p-1}x^{k}$
is the minimal polynomial of $\zeta _{p}$ in $%
\mathbb{Z}
\lbrack x]$. Following the same reasoning as in the case where $d=p$, we can
write $4m_{pq}(x)=f(x)^{2}\pm pqg(x)^{2}$ where $f(x),g(x)\in 
\mathbb{Z}
$. The rest of the problem is solved in a similar fashion as well.

\bigskip

Using some interesting approximation methods and quadratic number fields,
Ireland \& Rosen \cite{InR} show that $x^{2}-dy^{2}=1$ has \emph{infinitely
many solutions} for any square-free integer $d$ (including $d=2$), and that
every solution has the form $\pm (x_{n},y_{n})$ where $x_{n}+\sqrt{d}%
y_{n}=(x_{1}+\sqrt{d}y_{1})^{n}$ for some solution $(x_{1},y_{1})$ and
\thinspace $n\in 
\mathbb{Z}
$.

\bigskip 

\textbf{Acknowledgment }\bigskip \emph{Many thanks to Professor Dan Segal,
All-Souls College, Oxford, for his advice.}

\end{document}